\documentclass[reqno,oneside,11pt,english]{amsart}

\usepackage{amsmath,amsfonts,amssymb,amsthm,epsfig}
\usepackage{comment}

\usepackage{color}
\usepackage{cite}
\usepackage[pagebackref,colorlinks,citecolor=blue,linkcolor=blue]{hyperref}
\usepackage{graphicx}
\usepackage{titletoc,epsf}
\usepackage{latexsym,amsxtra,bbm}
\allowdisplaybreaks
\usepackage{ifthen}
\usepackage{float}

\newtheorem{thm}{Theorem}[section]
\newtheorem{lem}[thm]{Lemma}

\newtheorem{cor}[thm]{Corollary}
\newtheorem{prop}[thm]{Proposition}

\theoremstyle{definition}

\numberwithin{equation}{section}

{\qed\bigskip}

\newcommand{\R}{\mathbb R}
\newcommand{\Sn}{\mathbb S^{n-1}}
\newcommand{\Hn}{\mathcal H^{n-1}}

\newcommand{\dist}{\operatorname{dist}}
\newcommand{\supp}{\operatorname{supp}}

\newcommand{\E}{\mathcal E}
\newcommand{\A}{\mathcal A}

\usepackage{fancyhdr}
\newcommand{\papername}{\scriptsize {JIANG LI AND ZHUANG WANG}}
\newcommand{\papertitle}{\tiny {BOURGAIN--BREZIS--MIRONESCU FORMULA FOR BV FUNCTIONS ON ARBITRARY OPEN SETS AND APPLICATIONS}} 
\fancypagestyle{plain}{%
  \fancyhf{}
  \fancyhead[C]{\ifthenelse{\isodd{\value{page}}}{\papername}{\papertitle}}
  \fancyfoot[C]{\thepage}
  
}

\begin{document}

\title[]{Bourgain--Brezis--Mironescu formula for BV functions on arbitrary open sets and applications }
\author{}
\date{}
\subjclass[2020]{46E35, 26A45.}
\keywords{Bourgain--Brezis--Mironescu formula, functions of bounded variation, nonlocal functionals, $\Gamma$-convergence, compactness.}
\author{Jiang Li,  Zhuang Wang$^{*}$}


	\address[Jiang Li]{Key Laboratory of Computing and Stochastic Mathematics (Ministry of Education), School of Mathematics and Statistics, Hunan Normal University, Changsha, Hunan 410081, P.R. China.}
	\email{jiang\_li@hunnu.edu.cn; jiangli.math@qq.com}

	\address[Zhuang Wang]{Key Laboratory of Computing and Stochastic Mathematics (Ministry of Education), School of Mathematics and Statistics, Hunan Normal University, Changsha, Hunan 410081, P.R. China.}
	\email{zwang@hunnu.edu.cn}

	\date{\today}
  \thanks{$^*$Corresponding author: Zhuang Wang}
\begin{abstract}
In this paper, we establish a Bourgain--Brezis--Mironescu formula for BV functions on arbitrary open sets for a general class of radial mollifiers. Building on this formula, we identify the \(\Gamma\)-limit of the improved nonlocal energies  with respect to \(L^1\)-convergence as a constant multiple of the total variation. As a consequence, we obtain a characterization of BV functions in terms of the finiteness of the corresponding nonlocal energies. We also establish  local \(L^1\)-compactness  on arbitrary open sets and global \(L^1\)-compactness  on bounded open sets under a natural uniform integrability assumption. In particular, except for the global compactness result, all these results hold on arbitrary open sets, with no assumptions on boundedness, connectedness, or boundary regularity. Counterexamples show that for global compactness, both the boundedness of the domain and the uniform integrability assumption are essential. The proofs rely primarily on  one-dimensional
restrictions of BV functions and a concentration argument for radial mollifiers.
\end{abstract}
\maketitle
\tableofcontents

\section{Introduction}
\subsection{Bourgain--Brezis--Mironescu formula}
Let $\R^n$ be the $n$-dimensional Euclidean space with $n\ge2$.
The Bourgain--Brezis--Mironescu (BBM) formula provides nonlocal characterizations of
first-order Sobolev spaces and of functions of bounded variation on smooth domains. At the endpoint
$p=1$, D\'avila \cite{Davila} resolved an open problem posed by
Bourgain--Brezis--Mironescu \cite{BBM1} by establishing the corresponding limit
for BV functions.

More precisely, let $\{\rho_\lambda\}_{\lambda>0}\subset L^1(\R^n)$ be a family
of nonnegative radial functions satisfying
\begin{equation}\label{eq:kernel-mass}
	\int_{\R^n}\rho_\lambda(h)\,dh=1
	\qquad\text{for every }\lambda>0,
\end{equation}
and
\begin{equation}\label{eq:kernel-tail}
	\lim_{\lambda\to0^+}\int_{|h|>R}\rho_\lambda(h)\,dh=0
	\qquad\text{for every }R>0.
\end{equation}
If $\Omega\subset\R^n$ is a bounded Lipschitz domain, then
D\'avila \cite[Theorem~1]{Davila} proved that
\begin{equation}\label{eq:classical-bbm}
	\lim_{\lambda\to0^+}
	\int_\Omega\int_\Omega
	\frac{|u(x)-u(y)|}{|x-y|}\rho_\lambda(x-y)\,dy\,dx
	=K_{1,n}|Du|(\Omega)
\end{equation}
holds for every $u\in BV(\Omega)$, where
\begin{equation}\label{eq:Kn}
	K_{1,n}
	:=\frac{1}{\Hn(\Sn)}\int_{\Sn}|\sigma\cdot e|\,d\Hn(\sigma)
\end{equation}
and $e\in\Sn$ is any fixed unit vector. By rotational invariance, the value of
$K_{1,n}$ is independent of the choice of $e$.

Following the celebrated works of Bourgain--Brezis--Mironescu
\cite{BBM1,BBM2}, numerous further characterizations of Sobolev spaces have been
obtained by means of BBM-type nonlocal quantities; see, for example,
\cite{BourgainNguyen,Nguyen2006,Nguyen2008,Ponce2004}. The BBM formula has also
been extended to Orlicz and generalized Orlicz spaces
\cite{AlbericoEtAl2020,FernandezBonderSalort2019,FerreiraHastoRibeiro2020,YangYangYuan2019},
variable exponent spaces \cite{HastoRibeiro2017,FerrariSquassina2022}, magnetic
Sobolev and BV spaces \cite{SquassinaVolzone2016,PinamontiSquassinaVecchi2019},
anisotropic settings \cite{Ludwig2014,NguyenSquassina2019}, Riemannian manifolds
\cite{KreumlMordhorst2019}, metric measure spaces
\cite{DiMarinoSquassina2019,Gorny2022,LahtiPinamontiZhou2024}, and ball Banach
function spaces \cite{DaiEtAl2023,ZhuYangYuan2023}. Related limiting problems for
Besov spaces were studied in \cite{KolyadaLerner2005,Triebel2011}. Further
variants of the BBM limit can be found in
\cite{BraidesSolci2025,BrezisNguyen2016,
	BrezisSeegerVanSchaftingenYung2022,ClarosPerez2026,DominguezMilman2023,
	GennaioliStefani2025,LeoniSpector2011,LeoniSpector2014}.

A complementary direction is to relax the geometric assumptions on the domain.
Bal--Mohanta--Roy \cite{BalMohantaRoy2020} proved that the classical BBM formula
remains valid for $W^{1,p}$ functions on $W^{1,p}$-extension domains.
Drelichman--Dur\'an \cite{DrelichmanDuran} subsequently obtained  the BBM formula on arbitrary bounded domains by considering the improved fractional Sobolev energies. More precisely, if
$\Omega\subset\R^n$ is a bounded domain, $1<p<\infty$, $0<\tau<1$, and
$u\in W^{1,p}(\Omega)$, then
\[
\lim_{s\to1^-}(1-s)
\int_\Omega
\int_{B(x,\tau\dist(x,\partial\Omega))}
\frac{|u(x)-u(y)|^p}{|x-y|^{n+sp}}\,dy\,dx
=K_{p,n}\int_\Omega |\nabla u(x)|^p\,dx,
\]
where $K_{p,n}>0$ depends only on $p$ and $n$.

Mohanta \cite[Theorem 1]{Mohanta} later removed the boundedness assumption and treated the
more general $W_q^{s,p}$ framework, thereby answering the question raised in
\cite{BrazkeSchikorraYung2023} concerning the asymptotic behavior of the
corresponding seminorms. His result includes the endpoint $p=q=1$. At this
endpoint, however, the exact limit formula is proved under the assumption
$u\in W^{1,1}(\Omega)$, and it does not yield the corresponding formula for 
$u\in BV(\Omega)$. Moreover, the improved fractional Sobolev energies of
Drelichman--Dur\'an and Mohanta involve the specific mollifiers and do not treat the general radial mollifiers satisfying
\eqref{eq:kernel-mass} and \eqref{eq:kernel-tail}.

The first purpose of this paper is to address both of these points. We establish a
 BBM formula for BV functions on arbitrary open sets, using 
general radial mollifiers satisfying \eqref{eq:kernel-mass} and
\eqref{eq:kernel-tail}. Let  $\Omega\subset\R^n$ be an open set.
For $0<\tau<1$ and $u\in L^1(\Omega)$, define
\begin{equation}\label{eq:energy}
	\E_{\lambda,\tau}^{\rho}(u,\Omega)
	:=\int_\Omega\int_{B(x,\tau\dist(x,\partial\Omega))}
	\frac{|u(x)-u(y)|}{|x-y|}\rho_\lambda(x-y)\,dy\,dx.
\end{equation}

Our first main result is the following.

\begin{thm}\label{thm:main}
	Let $\Omega\subset\R^n$ be an open set, let $0<\tau<1$, and let 
	$\{\rho_\lambda\}_{\lambda>0}$ be a family of nonnegative radial functions satisfying \eqref{eq:kernel-mass} and
	\eqref{eq:kernel-tail}. Then, for every $u\in BV(\Omega)$,
	\begin{equation}\label{eq:main}
		\lim_{\lambda\to0^+}\E_{\lambda,\tau}^{\rho}(u,\Omega)
		=K_{1,n}|Du|(\Omega).
	\end{equation}
The constant $K_{1,n}$ is given by \eqref{eq:Kn}. 
\end{thm}
The use of the improved fractional Sobolev energy in \eqref{eq:energy} is in general essential. The energy in \eqref{eq:classical-bbm} does not yield a BBM formula on arbitrary open sets. The corresponding slit-domain counterexample is presented  in Section~\ref{subsec:slit-counterexample}. 

As a direct consequence of Theorem~\ref{thm:main}, by using specific mollifiers, we obtain the following BBM formula for BV functions. For
$u\in W^{1,1}(\Omega)$, it coincides with the endpoint case
$p=q=1$ of \cite[Theorem~1]{Mohanta}.

\begin{cor}\label{cor:main}
	Let $\Omega\subset\R^n$ be an open set and let $0<\tau<1$. Then, for every $u\in BV(\Omega)$,
	\begin{equation}\label{eq:main-1}
			\lim_{s\to1^-}(1-s)
		\int_\Omega
		\int_{B(x,\tau\dist(x,\partial\Omega))}
		\frac{|u(x)-u(y)|}{|x-y|^{n+s}}\,dy\,dx
		=K_{n}|Du|(\Omega),
	\end{equation}
	where $K_n=\Hn(\Sn)\,K_{1,n}$.
\end{cor}

\subsection{$\Gamma$-convergence}

The classical BBM formula is a pointwise statement. For each fixed $u$, the nonlocal energy converges to the corresponding
first-order energy \cite{BBM1,Davila}. However, for variational problems, this is not
sufficient, since minimizing sequences generally depend on the nonlocal
scale. The appropriate framework is therefore $\Gamma$-convergence. The main advantage of $\Gamma$-convergence is its stability for variational
problems. Combined with compactness, it allows one to pass from minimizers of
the nonlocal energies to minimizers of the limiting functional, see \cite{DalMaso,Braides}.

Therefore, it is the natural framework for the variational form of the BBM limit. The variational viewpoint for BBM-type energies was developed by Ponce
\cite{Ponce2004} and further studied by Nguyen, including  BV functions
\cite{Nguyen2007Gamma,Nguyen2011Gamma}. Related $\Gamma$-convergence results
for nonlocal perimeter functionals, together with compactness and stability
of minimizers, were obtained by Ambrosio--De Philippis--Martinazzi
\cite{AmbrosioDePhilippisMartinazzi2011}. More recently,
Gennaioli--Stefani \cite{GennaioliStefani2025} established sharp conditions
for pointwise and $\Gamma$-convergence, as well as compactness, for general
nonnegative kernels in the whole space $\mathbb R^n$.

Thus, we next turn to the variational behavior of the nonlocal energies on arbitrary open sets. For $\lambda>0$, define
$\mathcal F_\lambda:L^1(\Omega)\to[0,\infty]$ by
\[
 \mathcal F_\lambda(u):=\E_{\lambda,\tau}^{\rho}(u,\Omega),
\]
and define $\mathcal F:L^1(\Omega)\to[0,\infty]$ by
\begin{equation}\label{eq:Gamma-limit-functional}
 \mathcal F(u):=
 \begin{cases}
 K_{1,n}|Du|(\Omega),& u\in BV(\Omega),\\
 +\infty,&u\in L^1(\Omega)\setminus BV(\Omega).
 \end{cases}
\end{equation}

\begin{thm}\label{thm:gamma}
Let $\Omega\subset\R^n$ be an arbitrary open set, let $0<\tau<1$, and let
$\{\rho_\lambda\}_{\lambda>0}$ be a family of nonnegative radial functions satisfying \eqref{eq:kernel-mass} and
\eqref{eq:kernel-tail}. Then, for every sequence $\lambda_j\to0^+$ as $j\to\infty$, the
following statements hold.
\begin{enumerate}
\item[(i)] For every sequence $u_j\to u$ in $L^1(\Omega)$, then
\[
\mathcal F(u)
\le
\liminf_{j\to\infty}\mathcal F_{\lambda_j}(u_j).
\]

\item[(ii)] For every $u\in L^1(\Omega)$, there exists a sequence
$u_j\to u$ in $L^1(\Omega)$ such that
\[
\limsup_{j\to\infty}\mathcal F_{\lambda_j}(u_j)
\le
\mathcal F(u).
\]
\end{enumerate}
\end{thm}
In other words, for an arbitrary open set $\Omega$ and for every sequence $\lambda_j\to0^+$ as $j\to\infty$,  the $\Gamma$-limit of the functionals $\mathcal F_{\lambda_j}(u)=\E_{\lambda_j,\tau}^{\rho}(u,\Omega)$ with respect to $L^1$-convergence is
$$K_{1,n}|Du|(\Omega).$$ 

As an immediate consequence of Theorems~\ref{thm:main} and \ref{thm:gamma}, we obtain
a characterization of BV functions.

\begin{cor}\label{cor:characterization}
Let $\Omega\subset\R^n$ be an arbitrary open set, let $0<\tau<1$, let
$u\in L^1(\Omega)$, and let
$\{\rho_\lambda\}_{\lambda>0}$ be a family of nonnegative radial functions satisfying \eqref{eq:kernel-mass} and
\eqref{eq:kernel-tail}. Then
\[
 u\in BV(\Omega)
 \quad\Longleftrightarrow\quad
 \liminf_{\lambda\to0^+}\E_{\lambda,\tau}^{\rho}(u,\Omega)<\infty.
\]
In this case, we have
\begin{equation}\label{lower-to-limit}
\lim_{\lambda\to0^+}\E_{\lambda,\tau}^{\rho}(u,\Omega)
=K_{1,n}|Du|(\Omega).
\end{equation}
\end{cor}

\subsection{Local and global compactness}
Compactness is another fundamental component of BBM theory.
Bourgain--Brezis--Mironescu \cite{BBM1} and Ponce \cite{PonceJEMS2004} proved that if $\Omega$ is a bounded Lipschitz domain, the uniform boundedeness of $L^p$ and nonlocal energy
 yields strong $L^p$ compactness, with limits in $W^{1,p}(\Omega)$ for $p>1$
and in $BV(\Omega)$ for $p=1$.  
More recently, compactness has been studied in substantially more general
settings. Gennaioli--Stefani \cite{GennaioliStefani2025} treated more general kernels in the whole space, while
Lin--Yang--Yuan--Zhang \cite{LinYangYuanZhang2026} established a compactness
theorem for Sobolev spaces associated with ball Banach function spaces on
bounded Lipschitz domains. Related BBM characterizations in ball Banach
function spaces and on more general underlying spaces can be found in
\cite{ZhuYangYuan2023,HuLiYangYuan2025}.

Therefore, this leads to the compactness results on arbitrary open sets addressed in this paper. The first compactness result
is local and therefore needs no geometric hypothesis on the boundary.

\begin{thm}\label{thm:compactness-local}
Let $\Omega\subset\R^n$ be an arbitrary open set, let $\lambda_j\to0^+$ as $j\to\infty$, and let
$u_j\in L^1(\Omega)$ satisfy
\begin{equation}\label{eq:compactness-assumption}
 \sup_{j\ge1}\left(
 \|u_j\|_{L^1(\Omega)}
 +\E_{\lambda_j,\tau}^{\rho}(u_j,\Omega)
 \right)<\infty.
\end{equation}
Then $(u_j)$ is relatively compact in $L^1_{\mathrm{loc}}(\Omega)$. More precisely,
there exist a subsequence $(u_{j_k})$  and a function $u\in BV(\Omega)$ such that
\[
 u_{j_k}\to u\qquad\text{in }L^1_{\mathrm{loc}}(\Omega),
\]
and
\begin{equation}\label{eq:compactness-liminf}
 K_{1,n}|Du|(\Omega)
 \le\liminf_{j_k\to\infty}\E_{\lambda_{j_k},\tau}^{\rho}(u_{j_k},\Omega).
\end{equation}
\end{thm}

On a bounded open set, we have the following global compactness result.
\begin{thm}\label{thm:compactness}
Let $\Omega\subset\R^n$ be bounded and open, let $\lambda_j\to0^+$ as $j\to\infty$, and let
$u_j\in L^1(\Omega)$ satisfy \eqref{eq:compactness-assumption}. Assume in addition
that $(u_j)$ is uniformly integrable in $L^1(\Omega)$. Then $(u_j)$ is relatively compact in
$L^1(\Omega)$. More precisely, there exist a subsequence  $(u_{j_k})$  and
$u\in BV(\Omega)$ such that
\[
 u_{j_k}\to u\qquad\text{in }L^1(\Omega),
\]
and \eqref{eq:compactness-liminf} holds.
\end{thm}

For uniform integrability, we refer to Section~\ref{unform-int}. Counterexamples showing that   neither the boundedness of the domain nor the uniform integrability assumption in Theorem~\ref{thm:compactness} can be removed are provided in Sections~\ref{subsec:unbounded-counterexample} and \ref{subsec:ui-counterexample}.

The paper is organized as follows. In Section~\ref{sec:preliminaries}, we introduce
some preliminary concepts, a concentration lemma for radial mollifiers, and a convergence lemma for 
uniform integrability. In
Section~\ref{sec:bbm-formula}, we prove Theorem~\ref{thm:main} and
Corollary~\ref{cor:main}. Section~\ref{sec:gamma} is devoted to the proof of Theorem~\ref{thm:gamma} and
Corollary~\ref{cor:characterization}. In Section~\ref{sec:compactness}, we prove the
two compactness theorems, Theorems~\ref{thm:compactness-local} and
\ref{thm:compactness}.
Finally, Section~\ref{sec:counterexample} collects three aforementioned counterexamples. 

\section{Preliminaries}\label{sec:preliminaries}
Throughout the paper, let $\Omega\subset\mathbb R^n$, $n\ge2$, be an open set.  For $x\in\Omega$, we write
\[
d_\Omega(x)
:=
\operatorname{dist}(x,\partial\Omega)
=
\inf_{y\in\partial\Omega}|x-y|,
\]
with the convention $d_{\R^n}\equiv+\infty$ when $\Omega=\R^n$.
The open Euclidean ball centered at $x\in\mathbb R^n$ with radius $r>0$ is denoted by
$
B(x,r).
$
We write $\mathcal L^n$ for the $n$-dimensional Lebesgue measure and $\mathcal H^{n-1}$ for the $(n-1)$-dimensional Hausdorff measure. In this paper, unless otherwise stated,
$\{\rho_\lambda\}_{\lambda>0}$ denotes a family of nonnegative radial functions satisfying
\eqref{eq:kernel-mass} and \eqref{eq:kernel-tail}.
\subsection{Functions of bounded variation and directional variation} The theory of BV functions we rely on can be found in monographs \cite{AFP,EvansGariepy}. 
A function $u\in L^1(\Omega)$ is of bounded variation, denoted by $u\in BV(\Omega)$, if its distributional gradient $Du=(D_1u,\ldots,D_nu)$ is a finite $\R^n$-valued Radon measure on $\Omega$. Its variation measure is denoted by $|Du|$ and its total variation $|Du|(\Omega)$ can be defined by
\begin{equation}\label{eq:BVdual}
 |Du|(\Omega)=\sup\left\{\int_\Omega u\,\operatorname{div}\phi\,dx:
 \phi\in C_c^1(\Omega;\R^n),\ |\phi|\le1\right\}.
\end{equation}
The distributional gradient of $u\in BV(\Omega)$ is
\begin{equation}\label{eq:polar}
 Du=\nu_u|Du|,\qquad |\nu_u(x)|=1\quad\text{for }|Du|\text{-a.e. }x\in\Omega.
\end{equation}
For $\sigma\in\Sn$, define the directional derivative measure
\begin{equation*}\label{eq:Dsig}
 D_\sigma u:=\sigma\cdot Du.
\end{equation*}
By \eqref{eq:polar}, we have
\begin{equation*}\label{eq:Dsigvar}
 |D_\sigma u|=|\sigma\cdot\nu_u|\,|Du|.
\end{equation*}
Therefore, the Fubini theorem and rotational invariance yield
\begin{align}\label{eq:sphere-var}
 \int_{\Sn}|D_\sigma u|(\Omega)\,d\Hn(\sigma)
 =\int_\Omega\int_{\Sn}|\sigma\cdot\nu_u(x)|\,d\Hn(\sigma)\,d|Du|(x)
 &=K_n|Du|(\Omega).
\end{align}

Moreover, when $n=1$, let $(a,b)\subset\R$ be an open interval and let
$f\in BV((a,b))$. Then it follows from \cite[Theorem 5.21]{EvansGariepy} that
\begin{equation}\label{eq:essvar}
|Df|((a,b))=\operatorname{ess}V_{a}^b\, f.
\end{equation}
Here $\operatorname{ess}V_{a}^b\, f$ is  the essential variation of $f$ on $(a,b)$, defined by
\[
\operatorname{ess}V_{a}^b\, f
:=
\sup\left\{
\sum_{j=1}^{m}
|f(t_{j+1})-f(t_j)|
\right\},
\]
where the supremum is taken over all finite partitions
$a<t_1<\cdots<t_{m+1}<b$
such that each $t_j$ is a point of approximate continuity of $f$.
For more details, see \cite[Section 5.10]{EvansGariepy}.
\subsection{One-dimensional restrictions of BV functions}\label{one-d-restriction}

Fix $\sigma\in\Sn$ and let $\sigma^\perp$ be the hyperplane orthogonal to $\sigma$, that is, 
\begin{equation*}\label{eq:sigperp}
\sigma^\perp:=\{z\in\R^n:z\cdot\sigma=0\}.
\end{equation*}
Let $\Omega_\sigma$ be the orthogonal projection of $\Omega$ on $\sigma^\perp$. 
Then for any $z\in\Omega_\sigma$, the section of $\Omega$ corresponding to $z$
\begin{equation*}\label{eq:slice-domain}
\Omega_z^\sigma
:=
\{t\in\R:z+t\sigma\in\Omega\}
\end{equation*}
is not empty. 
For any function $u: \Omega\to \mathbb R$ and any $z\in \Omega_\sigma$, define the function $u_z^\sigma: \Omega_z^\sigma\to \mathbb R$ by 
\begin{equation*}\label{eq:slice-fn}
u_z^\sigma(t)=u(z+t\sigma).
\end{equation*}

Then it follows from 
\cite[Theorem 3.107]{AFP}  that for $u\in BV(\Omega)$ and for
$ \mathcal H^{n-1}$-a.e. $z\in\Omega_\sigma$, the function $u_z^\sigma$ belongs to
$BV(\Omega_z^\sigma)$, and the signed measure $D_\sigma u$ coincides with $\Hn|_{\Omega_\sigma} \otimes Du_z^\sigma$. In particular,
for any bounded Borel function $\varphi: \Omega\to \mathbb R$,
\begin{equation*}\label{eq:slice-signed}
\int_\Omega \varphi\,dD_\sigma u
=
\int_{\Omega_\sigma}
\int_{\Omega_z^\sigma}
\varphi(z+t\sigma)\,dDu_z^\sigma(t)\,
d\Hn(z).
\end{equation*}
The corresponding disintegration of the total variation gives
\begin{equation}\label{eq:slice-variation}
 |D_\sigma u|(\Omega)
 =
 \int_{\Omega_\sigma}|Du_z^\sigma|(\Omega_z^\sigma)\,d\Hn(z).
\end{equation}
For more on one-dimensional restrictions of BV functions, we refer interested readers to \cite[Section 3.11]{AFP}.

\subsection{A concentration lemma for radial mollifiers}
\begin{lem}\label{lem:abelian}
Let $\sigma\in \mathbb S^{n-1}$ and $R>0$.  Let $0\leq A\leq M<\infty$. Assume that $g:(0,R)\to[0,\infty)$ is a bounded measurable function such that \[ 0\le g(r)\le M \qquad\text{for all }r\in(0,R), \] and suppose that \[ \lim_{r\to0^+}g(r)=A. \] 
If $\{\rho_\lambda\}_{\lambda>0}$ is a family of nonnegative radial functions satisfying \eqref{eq:kernel-mass} and \eqref{eq:kernel-tail}, then we have
\begin{equation}\label{eq:abelian-limit}
\lim_{\lambda\to0^+} \int_0^R \rho_\lambda(r\sigma)\,r^{n-1}g(r)\,dr = \frac{A}{\mathcal H^{n-1}(\mathbb S^{n-1})}
\end{equation}
and
\begin{equation}\label{eq:abelian-bounds}
0\le \int_0^R \rho_\lambda(r\sigma)\,r^{n-1}g(r)\,dr \le \frac{M}{\mathcal H^{n-1}(\mathbb S^{n-1})}. 
\end{equation}
\end{lem}
\begin{proof}
Since $\rho_\lambda$ is radial and satisfies \eqref{eq:kernel-mass},
by applying the polar coordinates, we have that for every
$\sigma\in\mathbb S^{n-1}$,
\begin{equation}\label{unit}
\int_0^\infty
\rho_\lambda(r\sigma)\,r^{n-1}\,dr
=
\frac1{\mathcal H^{n-1}(\mathbb S^{n-1})}.
\end{equation}
Consequently, 
\[ 0\le \int_0^R \rho_\lambda(r\sigma)\,r^{n-1}g(r)\,dr \le M\int_0^\infty \rho_\lambda(r\sigma)\,r^{n-1}\,dr = \frac{M}{\mathcal H^{n-1}(\mathbb S^{n-1})}, \]
which proves \eqref{eq:abelian-bounds}.

It is left to prove \eqref{eq:abelian-limit}. Fix $\varepsilon>0$. Since
$g(r)\to A$ as $r\to0^+$, there exists
$\delta\in(0,R)$ such that
\[
|g(r)-A|<\varepsilon
\qquad\text{for }0<r<\delta.
\]
Combining this with \eqref{unit} yields
\begin{align*}
&\left|
\int_0^R
\rho_\lambda(r\sigma)\,r^{n-1}g(r)\,dr
-
\frac{A}{\mathcal H^{n-1}(\mathbb S^{n-1})}
\right|
\\
\le&
\int_0^R
\rho_\lambda(r\sigma)\,r^{n-1}|g(r)-A|\,dr
+
A\int_R^\infty
\rho_\lambda(r\sigma)\,r^{n-1}\,dr
\\
\le&
\varepsilon
\int_0^\delta
\rho_\lambda(r\sigma)\,r^{n-1}\,dr
+
(M+A)
\int_\delta^R
\rho_\lambda(r\sigma)\,r^{n-1}\,dr
+
A\int_R^\infty
\rho_\lambda(r\sigma)\,r^{n-1}\,dr
\\
\le&
\frac{\varepsilon}{\mathcal H^{n-1}(\mathbb S^{n-1})}
+
(M+A)
\int_\delta^\infty
\rho_\lambda(r\sigma)\,r^{n-1}\,dr.
\end{align*}
Moreover, 
\[
\int_\delta^\infty
\rho_\lambda(r\sigma)\,r^{n-1}\,dr
=
\frac1{\mathcal H^{n-1}(\mathbb S^{n-1})}
\int_{\mathbb R^n\setminus B_\delta(0)}
\rho_\lambda(x)\,dx.
\]
This quantities tend to zero as $\lambda\to0^+$, because $\{\rho_\lambda\}_{\lambda>0}$ satisfies \eqref{eq:kernel-tail}.
Therefore
\[
\limsup_{\lambda\to0^+}
\left|
\int_0^R
\rho_\lambda(r\sigma)\,r^{n-1}g(r)\,dr
-
\frac{A}{\mathcal H^{n-1}(\mathbb S^{n-1})}
\right|
\le
\frac{\varepsilon}{\mathcal H^{n-1}(\mathbb S^{n-1})}.
\]
Since $\varepsilon>0$ is arbitrary, we obtain \eqref{eq:abelian-limit} and completes the proof.
\end{proof}

\subsection{Uniform integrability and a convergence lemma}\label{unform-int}
A sequence $(u_j)\subset L^1(\Omega)$ is said to be \emph{uniformly integrable}
if for every $\varepsilon>0$, there exists $\delta>0$ such that, for every
measurable set $E\subset\Omega$ with
\[
\mathcal L^n(E)<\delta,
\]
one has
\[
\int_E |u_j(x)|\,dx<\varepsilon
\quad\text{for all}\ j.
\]

\begin{lem}\label{lem:local-to-global}
Let $\Omega\subset\R^n$ be bounded and open. Suppose that
$u_j\to u$ in $L^1_{\mathrm{loc}}(\Omega)$ and that $(u_j)$ is uniformly
integrable in $L^1(\Omega)$. Then $u\in L^1(\Omega)$ and
\[
 u_j\to u\qquad\text{in }L^1(\Omega).
\]
\end{lem}

This is a direct application of the Vitali convergence theorem; see
\cite[Exercise~1.18]{AFP}. Indeed, since $\Omega$ is bounded, local $L^1$
convergence implies convergence in measure on $\Omega$, and the  uniform
integrability is precisely the equiintegrability hypothesis in Vitali's theorem. See \cite[Definition~1.26]{AFP} for the definition of equiintegrability.

\section{Bourgain--Brezis--Mironescu formula for BV functions }\label{sec:bbm-formula}
In this section, we mainly prove Theorem~\ref{thm:main} and Corollary~\ref{cor:main}. Before the proofs, we first give some auxiliary results.

\begin{lem}\label{lem:1d}
Let $I\subset\R$ be an arbitrary open set, let $v\in BV(I)$, and let $r>0$. Set
\begin{equation*}\label{eq:ArI}
 A_r(I):=\{t\in I:[t,t+r]\subset I\}.
\end{equation*}
Then
\begin{equation*}\label{eq:1d-trans}
 \int_{A_r(I)}|v(t+r)-v(t)|\,dt\le r|Dv|(I).
\end{equation*}
\end{lem}

\begin{proof}
Let $J_v$ be the set of points at which the precise representative of $v$ is not approximately continuous. Then $\mathcal L^1(J_v)=0$ (cf. \cite[page 238]{EvansGariepy}). Hence
\[
 N_r:=J_v\cup(J_v-r)
\]
is a null set. Fix $t\in A_r(I)\setminus N_r$. Then both $t$ and $t+r$ are points of approximate continuity of $v$. Since $[t,t+r]\subset I$ and this interval is connected, it is contained in one connected component of $I$. Therefore there exists $\varepsilon>0$ such that
\[
 [t-\varepsilon,t+r+\varepsilon]\subset I.
\]
By \eqref{eq:essvar} and the definition of the essential variation on $(t-\varepsilon,t+r+\varepsilon)$, we obtain
\[
 |v(t+r)-v(t)|
 \le \operatorname{ess}V_{t-\varepsilon}^{t+r+\varepsilon}\, v
 =|Dv|((t-\varepsilon,t+r+\varepsilon)).
\]
Since $|Dv|$ is a Radon measure, letting $\varepsilon\to 0$ yields that for $\mathcal L^1$-a.e. $t\in A_r(I)$,
\begin{equation}\label{eq:1d-essvar}
 |v(t+r)-v(t)|\le |Dv|([t,t+r]).
\end{equation}

By integrating \eqref{eq:1d-essvar} with respect to $t$ and applying the Fubini theorem, we obtain
\begin{align*}
 \int_{A_r(I)}|v(t+r)-v(t)|\,dt
 &\le\int_{A_r(I)}|Dv|([t,t+r])\,dt=\int_{A_r(I)}\int_{[t, t+r]} d|Dv|(q)\,dt\\
 &=\int_I\mathcal L^1\bigl(\{t\in A_r(I):q\in[t,t+r]\}\bigr)\,d|Dv|(q).
\end{align*}
For fixed $q$, since the condition $q\in[t,t+r]$ implies $q-r\leq t\leq q$, we have
\[
 \mathcal L^1\bigl(\{t\in A_r(I):q\in[t,t+r]\}\bigr)\leq r.
\]
Consequently
\[
 \int_{A_r(I)}|v(t+r)-v(t)|\,dt\le r|Dv|(I),
\]
which completes the proof.
\end{proof}

For $\sigma\in\Sn$ and $r>0$, define the $n$-dimensional admissible set
\begin{equation*}\label{eq:mathcalAr}
 \A_r:=\{y\in\Omega:r<\tau d_\Omega(y)\}
\end{equation*}
and
\begin{equation*}\label{eq:Gsig}
 G_\sigma(r):=\int_{\A_r}|u(y+r\sigma)-u(y)|\,dy.
\end{equation*}
If $y\in\A_r$, then $r<d_\Omega(y)$, and hence, $y+\theta\sigma\in B(y,d_\Omega(y))\subset\Omega$ for every $0\le\theta\le r$. 

For $z\in \Omega_\sigma$, we define the  section of $\A_r$  corresponding to $z$ 
\begin{equation*}\label{eq:Arz}
 \A_{r,z}^\sigma
 :=\{t\in\Omega_z^\sigma:r<\tau d_\Omega(z+t\sigma)\}.
\end{equation*}
Then we have the following implications
\begin{equation}\label{eq:slice-Ar}
 t\in\A_{r,z}^\sigma\quad\Longleftrightarrow\quad z+t\sigma\in\A_r.
\end{equation}

\begin{prop}\label{prop:dir-upper}
Let $u\in BV(\Omega)$, $\sigma\in\Sn$, and $r>0$. Then
\begin{equation}\label{eq:g-upper}
 0\le\frac{G_\sigma(r)}r\le |D_\sigma u|(\Omega).
\end{equation}
\end{prop}

\begin{proof}
Write $y=z+t\sigma$ with $z\in\Omega_\sigma$ and $t\in\R$. Then using \eqref{eq:slice-Ar} and the Fubini theorem, we have
\begin{align}\label{eq:G-slices}
 G_\sigma(r)
 &=\int_{\Omega_\sigma}\int_{\A_{r,z}^\sigma}
 |u_z^\sigma(t+r)-u_z^\sigma(t)|\,dt\,d\Hn(z).
\end{align}
Recall from Section \ref{one-d-restriction} that $u_z^\sigma\in BV(\Omega_z^\sigma)$ for $\Hn$-a.e. $z\in \Omega_\sigma$.
We claim that
\begin{equation}\label{eq:Arz-inclusion}
 \A_{r,z}^\sigma\subset A_r(\Omega_z^\sigma).
\end{equation}
Indeed, let $t\in\A_{r,z}^\sigma$. Then
\[
 r<\tau d_\Omega(z+t\sigma)<d_\Omega(z+t\sigma).
\]
For every $s\in[t,t+r]$, since
\[
 |z+s\sigma-(z+t\sigma)|=|s-t|\leq  r<d_\Omega(z+t\sigma),
\]
we have $z+s\sigma\in B(z+t\sigma,d_\Omega(z+t\sigma))\subset\Omega$. Hence $s\in\Omega_z^\sigma$ for all $s\in[t,t+r]$, which gives \eqref{eq:Arz-inclusion}.

By \eqref{eq:Arz-inclusion} and Lemma~\ref{lem:1d}, for $\Hn$-a.e. $z\in \Omega_\sigma$, 
\begin{align*}
 \int_{\A_{r,z}^\sigma}|u_z^\sigma(t+r)-u_z^\sigma(t)|\,dt
 \le\int_{A_r(\Omega_z^\sigma)}|u_z^\sigma(t+r)-u_z^\sigma(t)|\,dt
\le r|Du_z^\sigma|(\Omega_z^\sigma).
\end{align*}
Integrating with respect to $z$ over $\Omega_\sigma$, we deduce from \eqref{eq:slice-variation} and \eqref{eq:G-slices} that
\[
 G_\sigma(r)
 \le r\int_{\Omega_\sigma}|Du_z^\sigma|(\Omega_z^\sigma)\,d\Hn(z)
 =r|D_\sigma u|(\Omega),
\]
which proves \eqref{eq:g-upper}. 
\end{proof}

\begin{prop}\label{prop:dir-limit}
Let $u\in BV(\Omega)$ and $\sigma\in\Sn$. Then
\begin{equation}\label{eq:dir-limit}
 \lim_{r\to0^+}\frac{G_\sigma(r)}r
 =|D_\sigma u|(\Omega).
\end{equation}
\end{prop}

\begin{proof}
It follows immediately from Proposition~\ref{prop:dir-upper} that
\begin{equation}\label{eq:limsup-dir}
 \limsup_{r\to0^+}\frac{G_\sigma(r)}r\le |D_\sigma u|(\Omega).
\end{equation}
It remains to prove the reverse inequality.

Fix $\varphi\in C_c^1(\Omega)$ with $|\varphi|\le1$, and let
\[
 K:=\supp\varphi,\qquad a:=\dist(K,\partial\Omega)>0.
\]
If $0<r<\tau a$, then $K\subset\A_r$. Choose $0<\rho<\tau a/2$. For $0<r<\rho$, both $K$ and $K+r\sigma$ are compactly contained in $\Omega$.

Define
\begin{equation*}\label{eq:Jr}
 J_r:=\int_K\frac{u(y+r\sigma)-u(y)}r\,\varphi(y)\,dy.
\end{equation*}
By a change of variable with $x=y+r\sigma$, we have
\[
 \int_Ku(y+r\sigma)\varphi(y)\,dy
 =\int_{K+r\sigma}u(x)\varphi(x-r\sigma)\,dx.
\]
Since $\varphi(x-r\sigma)$ is supported in $K+r\sigma$ and $\varphi(x)$ is supported in $K$, both of which lie in $\Omega$ for  $0<r<\rho$, we can rewrite
\begin{equation}\label{eq:Jr-change}
 J_r=\int_\Omega u(x)\frac{\varphi(x-r\sigma)-\varphi(x)}r\,dx.
\end{equation}

Set
\[
 q_r(x):=\frac{\varphi(x-r\sigma)-\varphi(x)}r.
\]
Since $\varphi\in C_c^1(\Omega)$, 
we have $q_r(x)\to-\partial_\sigma\varphi(x)$  as $r\to0^+$ for every point $x\in \Omega$. Moreover, for $0<r<\rho$, we have
\[
 \supp q_r\subset \{x\in\Omega:\dist(x,K)\le\rho\}\Subset\Omega,
\]
and for every point $x\in \Omega$,
\[
 |q_r(x)|\le\|\partial_\sigma\varphi\|_{L^\infty(\Omega)}<\infty.
\]
Since $u\in L^1(\Omega)$,  applying the Lebesgue dominated convergence theorem in \eqref{eq:Jr-change} gives
\begin{equation}\label{eq:Jr-limit}
 \lim_{r\to0^+}J_r
 =-\int_\Omega u\,\partial_\sigma\varphi\,dx
 =\int_\Omega\varphi\,dD_\sigma u.
\end{equation}

On the other hand, because $K\subset\A_r$ and $|\varphi|\le1$, we obtain
\begin{align*}
 |J_r|
 \le\frac1r\int_K|u(y+r\sigma)-u(y)|\,dy
 &\le\frac{G_\sigma(r)}r.
\end{align*}
Passing to the lower limit and using \eqref{eq:Jr-limit}, we get
\[
 \left|\int_\Omega\varphi\,dD_\sigma u\right|
 \le\liminf_{r\to0^+}\frac{G_\sigma(r)}r.
\]
Taking the supremum over all $\varphi\in C_c^1(\Omega)$ with $|\varphi|\le1$, by the dual characterization of the total variation of the signed measure $D_\sigma u$, we get
\begin{equation*}\label{eq:liminf-dir} 
 |D_\sigma u|(\Omega)
 \le\liminf_{r\to0^+}\frac{G_\sigma(r)}r.
\end{equation*}
Combining this with \eqref{eq:limsup-dir}  proves \eqref{eq:dir-limit}.
\end{proof}

We now turn to the improved fractional Sobolev energy  $\E_{\lambda,\tau}^{\rho}(u,\Omega)$. Let $u\in L^1(\Omega)$,  $0<\tau<1$, and  $\lambda>0$. By a change of variable with $h=x-y$, we have
\[
\E_{\lambda,\tau}^{\rho}(u,\Omega)
=\int_\Omega\int_{|h|<\tau d_\Omega(y)}
\frac{|u(y+h)-u(y)|}{|h|}\rho_\lambda(h)\,dh\,dy.
\]
Since $\rho_\lambda$ is radial, applying the polar coordinates $h=r\sigma$ and using the Fubini theorem, we obtain
\begin{equation}\label{eq:polar-energy}
	\E_{\lambda,\tau}^{\rho}(u,\Omega)
	=\int_{\Sn}\int_0^\infty
	\rho_\lambda(r\sigma)r^{n-1}\frac{G_\sigma(r)}{r}\,dr\,d\Hn(\sigma).
\end{equation}

\begin{prop}\label{prop:Ele}
	Let $u\in BV(\Omega)$. Then
	$$\E_{\lambda,\tau}^{\rho}(u,\Omega)\leq K_{1, n}|Du|(\Omega).$$
\end{prop}
\begin{proof}
	Since $\rho_\lambda$ is radial and satisfies \eqref{eq:kernel-mass}, we know that \eqref{unit} holds for every $\sigma\in \Sn$, i.e.,
	$$\int_{0}^{\infty} \rho(r\sigma)r^{n-1}\, dr =\frac{1}{\Hn(\Sn)}.$$
By applying Proposition~\ref{prop:dir-upper} to \eqref{eq:polar-energy}, we deduce from \eqref{eq:sphere-var} that
\begin{align*}
\E_{\lambda,\tau}^{\rho}(u,\Omega)&\leq  \int_{\Sn}\int_0^\infty 	\rho_\lambda(r\sigma)r^{n-1} |D_\sigma u|(\Omega)\, dr\,d\Hn(\sigma)\\
&=\frac{1}{\Hn(\Sn)}  \int_{\Sn}|D_\sigma u|(\Omega)\,d\Hn(\sigma)=K_{1, n}|Du|(\Omega).
\end{align*}	
\end{proof}

\begin{prop}\label{prop:smallscale}
If $u\in BV(\Omega)$, then for every $R>0$,
\[
 \lim_{\lambda\to0^+}
 \int_{\Sn}\int_0^R
 \rho_\lambda(r\sigma)r^{n-1}\frac{G_\sigma(r)}{r}\,dr\,d\Hn(\sigma)
 =K_{1,n}|Du|(\Omega).
\]
\end{prop}

\begin{proof}
Fix $\sigma\in\Sn$. Propositions~\ref{prop:dir-upper} and \ref{prop:dir-limit} give
\[
 0\le \frac{G_\sigma(r)}r\le |D_\sigma u|(\Omega)\quad \text{and}
 \quad
 \lim_{r\to0^+}\frac{G_\sigma(r)}r=|D_\sigma u|(\Omega).
\]
Applying Lemma~\ref{lem:abelian} to $g(r)=G_\sigma(r)/r$,  we obtain
\[
 \lim_{\lambda\to0^+}\int_0^R
 \rho_\lambda(r\sigma)r^{n-1}\frac{G_\sigma(r)}r\,dr
 =\frac{|D_\sigma u|(\Omega)}{\Hn(\Sn)}
\]
and 
\[
 \int_0^R
 \rho_\lambda(r\sigma)r^{n-1}\frac{G_\sigma(r)}r\,dr
 \le \frac{|D_\sigma u|(\Omega)}{\Hn(\Sn)}.
\]
Hence \eqref{eq:sphere-var} and the Lebesgue dominated convergence theorem give
\[
 \lim_{\lambda\to0^+}
 \int_{\Sn}\int_0^R
 \rho_\lambda(r\sigma)r^{n-1}\frac{G_\sigma(r)}r\,dr\,d\Hn(\sigma)
 =\frac{K_n}{\Hn(\Sn)}|Du|(\Omega)
 =K_{1,n}|Du|(\Omega),
\]
which completes the proof.
\end{proof}

Based on the previous preparations, we can now begin to prove Theorem~\ref{thm:main} and Corollary~\ref{cor:main}.
\begin{proof}[Proof of Theorem~\ref{thm:main}]
Fix $R>0$. By \eqref{eq:polar-energy},
\begin{align*}
 \E_{\lambda,\tau}^{\rho}(u,\Omega)
 &=\int_{\Sn}\int_0^R
 \rho_\lambda(r\sigma)r^{n-1}\frac{G_\sigma(r)}r\,dr\,d\Hn(\sigma)\\
 &\quad+\int_{\Sn}\int_R^\infty
 \rho_\lambda(r\sigma)r^{n-1}\frac{G_\sigma(r)}r\,dr\,d\Hn(\sigma).
\end{align*}
The first term converges to $K_{1,n}|Du|(\Omega)$ by Proposition~\ref{prop:smallscale}.

For the second term, since $u\in BV(\Omega)\subset L^1(\Omega)$, for every $r>0$ and $\sigma\in\Sn$,
\begin{equation}\label{G-L1}
 G_\sigma(r)
 \le\int_{\A_r}|u(y+r\sigma)|\,dy+\int_{\A_r}|u(y)|\,dy
 \le2\|u\|_{L^1(\Omega)}.
\end{equation}
Hence
\begin{align*}
0&\le\int_{\Sn}\int_R^\infty
 \rho_\lambda(r\sigma)r^{n-1}\frac{G_\sigma(r)}r\,dr\,d\Hn(\sigma)\\
&\le \frac{2\|u\|_{L^1(\Omega)}}{R}
 \Hn(\Sn)\int_R^\infty\rho_\lambda(r\sigma)r^{n-1}\,dr\\
&=\frac{2\|u\|_{L^1(\Omega)}}{R}
 \int_{|h|>R}\rho_\lambda(h)\,dh,
\end{align*}
which tends to zero as $\lambda\to0^+$ by \eqref{eq:kernel-tail}. Therefore
\[
 \lim_{\lambda\to0^+}\E_{\lambda,\tau}^{\rho}(u,\Omega)
 =K_{1,n}|Du|(\Omega).
\]
This proves \eqref{eq:main} and completes the proof of Theorem~\ref{thm:main}.
\end{proof}

\begin{proof}[Proof of Corollary~\ref{cor:main}]
	Let $u\in BV(\Omega)$ and define
	\begin{equation*}\label{eq:energy-kernel}
		\E_{s,\tau}(u,\Omega)
		:=\int_\Omega\int_{B(x, \tau d_\Omega(x))}
		\frac{|u(x)-u(y)|}{|x-y|^{n+s}}\,dy\,dx, \quad 0<s<1.
	\end{equation*}
By the same reasoning as \eqref{eq:polar-energy}, we can rewrite
$$\E_{s,\tau}(u,\Omega)=\int_{\Sn}\int_0^\infty r^{-1-s}G_\sigma(r)\,dr\,d\Hn(\sigma).$$
Fix $R>0$, we have
\begin{align*}
	\E_{s,\tau}(u,\Omega)
	&=\int_{\Sn}\int_0^R r^{-1-s}G_\sigma(r)\,dr\,d\Hn(\sigma)\\
	&\quad +\int_{\Sn}\int_R^\infty r^{-1-s}G_\sigma(r)\,dr\,d\Hn(\sigma).
\end{align*}

Since the radial mollifiers $\{\rho_\lambda\}_{\lambda>0}$  defined by
$$\rho_{\lambda}(x)=\frac{\lambda}{R^\lambda\Hn(\Sn)} \frac{\mathbf1_{[0, R]}(|x|)}{|x|^{n-\lambda}}$$
satisfy \eqref{eq:kernel-mass} and \eqref{eq:kernel-tail}, by taking $\lambda=1-s$ and letting $s\to 1^{-}$, we deduce from Proposition~\ref{prop:smallscale} that
$$\lim_{s\to1^-} \frac{1-s}{R^{1-s}\Hn(\Sn)} \int_{\Sn}\int_0^R r^{-1-s}G_\sigma(r)\,dr\,d\Hn(\sigma)=K_{1,n}|Du|(\Omega).$$
Therefore,
\begin{equation}\label{eq:first-term}
\lim_{s\to1^-} {(1-s)} \int_{\Sn}\int_0^R r^{-1-s}G_\sigma(r)\,dr\,d\Hn(\sigma)=K_{n}|Du|(\Omega).
\end{equation}
Moreover, since \eqref{G-L1} holds for $u$, we obtain
\[
0\le(1-s)\int_{\Sn}\int_R^\infty r^{-1-s}G_\sigma(r)\,dr\,d\Hn(\sigma)
\le2\Hn(\Sn)\frac{1-s}{s}R^{-s}\|u\|_{L^1(\Omega)},
\]
which tends to zero as $s\to1^-$. Combining this with \eqref{eq:first-term} yields
$$\lim_{s\to1^{-}} (1-s) \E_{s,\tau}(u,\Omega)=K_{n}|Du|(\Omega),$$
which proves \eqref{eq:main-1} and completes the proof of Corollary ~\ref{cor:main}.
\end{proof}

\section{\texorpdfstring{$\Gamma$-convergence}{Gamma-convergence} and characterization of BV functions}\label{sec:gamma}
In this section, we prove Theorem~\ref{thm:gamma} and
Corollary~\ref{cor:characterization}. We begin with a limit result for smooth functions.

\begin{lem}\label{lem:smoothtrunc}
Let $\Omega\subset\R^n$ be bounded and open, let $R>0$, and let
$f\in C^1(\R^n)\cap W^{1,\infty}(\R^n)$. Then
\begin{equation}\label{eq:smoothtrunc}
 \lim_{\lambda\to0^+}
 \int_{\Omega}\int_{\left\{y\in \Omega:|x-y|<R\right\}}
 \frac{|f(x)-f(y)|}{|x-y|}\rho_\lambda(x-y)\,dy\,dx
 =K_{1,n}\int_\Omega|\nabla f(x)|\,dx.
\end{equation}
\end{lem}

\begin{proof}
By using polar coordinates with $y=x+r\sigma$, we can rewrite the double integral
in the left hand side of \eqref{eq:smoothtrunc} as
\[
 \int_\Omega\int_{\Sn}\int_0^R
 \mathbf1_\Omega(x+r\sigma)
 \frac{|f(x+r\sigma)-f(x)|}{r}
 \rho_\lambda(r\sigma)r^{n-1}\,dr\,d\Hn(\sigma)\,dx.
\]
For $x\in\Omega$ and $\sigma\in\Sn$, define
\[
 q_{x,\sigma}(r):=\mathbf1_\Omega(x+r\sigma)
 \frac{|f(x+r\sigma)-f(x)|}{r}.
\]
Since $\Omega$ is open, for each fixed $x$ the characteristic function is equal
to one for all sufficiently small $r$. Hence
\[
 \lim_{r\to0^+}q_{x,\sigma}(r)=|\nabla f(x)\cdot\sigma|.
\]
Moreover,
\[
 0\le q_{x,\sigma}(r)\le\|\nabla f\|_{L^\infty(\R^n)}=:M.
\]
Applying Lemma~\ref{lem:abelian} to $g(r)=q_{x,\sigma}(r)$ yields
\[
 \lim_{\lambda\to0^+}\int_0^R
 \rho_\lambda(r\sigma)r^{n-1}q_{x,\sigma}(r)\,dr
 =\frac{|\nabla f(x)\cdot\sigma|}{\Hn(\Sn)}
\]
and
\[
 \int_0^R
 \rho_\lambda(r\sigma)r^{n-1}q_{x,\sigma}(r)\,dr
 \le \frac{M}{\Hn(\Sn)}.
\]
Since $\Omega$ has finite measure and $\Sn$ has finite surface measure, the
Lebesgue dominated convergence theorem gives
\begin{align*}
 &\lim_{\lambda\to0^+}
 \int_{\Omega}\int_{\{y\in\Omega:|x-y|<R\}}
 \frac{|f(x)-f(y)|}{|x-y|}\rho_\lambda(x-y)\,dy\,dx\\
 &=\frac1{\Hn(\Sn)}
 \int_\Omega\int_{\Sn}|\nabla f(x)\cdot\sigma|\,d\Hn(\sigma)\,dx
 =K_{1,n}\int_\Omega|\nabla f(x)|\,dx.
\end{align*}
This gives \eqref{eq:smoothtrunc} and completes the proof.
\end{proof}

\begin{lem}\label{lem:mollify}
Let $\lambda_j\to0^+$ as $j\to\infty$ and let $u_j,u\in L^1(\Omega)$ satisfy
\[
 u_j\to u\qquad\text{in}\quad L^1_{\mathrm{loc}}(\Omega).
\]
Then, for every bounded open set $U\Subset\Omega$,
\begin{equation}\label{eq:local-liminf}
 K_{1,n}|Du|(U)
 \le \liminf_{j\to\infty}\E_{\lambda_j,\tau}^{\rho}(u_j,\Omega).
\end{equation}
\end{lem}

\begin{proof}
Set
\[
 L:=\liminf_{j\to\infty}\E_{\lambda_j,\tau}^{\rho}(u_j,\Omega).
\]
There is nothing to prove if $L=+\infty$. If $L<+\infty$,  passing to a subsequence without relabeling for brevity, we may assume
\begin{equation}\label{limit-le-L}
\E_{\lambda_j,\tau}^{\rho}(u_j,\Omega)\rightarrow L\quad \text{as}\ j\to\infty.
\end{equation}

Let $U\Subset\Omega$ be a bounded open set, and set
\[
 d:=\min\{1,\dist(U,\partial\Omega)\}>0,
 \qquad
 R:=\frac{\tau d}{8}.
\]
Choose a nonnegative standard mollifier
$\eta\in C_c^\infty(B(0,1))$ with $\int_{\R^n}\eta\,dx=1$, and set
\[
 \eta_\varepsilon(x):=\varepsilon^{-n}\eta(x/\varepsilon).
\]
Extend $u_j$ and $u$ by zero outside $\Omega$, denoting the extensions by
$\widetilde u_j$ and $\widetilde u$, respectively. For $0<\varepsilon<d/8$, define
\[
 u_{j,\varepsilon}:=\eta_\varepsilon*\widetilde u_j,
 \qquad
 u_\varepsilon:=\eta_\varepsilon*\widetilde u.
\]
Then by the standard properties of mollification (cf. \cite[Theorem 4.1]{EvansGariepy}), we have
$$u_{j,\varepsilon}, u_\varepsilon\in C^\infty(\R^n)\cap W^{1,\infty}(\R^n).$$
Moreover, for $x,y\in U$, 
\[
 |u_{j,\varepsilon}(x)-u_{j,\varepsilon}(y)|
 \le\int_{B(0,\varepsilon)}\eta_\varepsilon(z)
 |u_j(x-z)-u_j(y-z)|\,dz.
\]
Multiplying by
\[
\mathbf1_{\{y\in U: |x-y|<R\}}\frac{\rho_{\lambda_j}(x-y)}{|x-y|}
\]
and integrating over $U\times U$, the Fubini theorem yields
\begin{align}\label{Fubini-section4}
	&\int_U\int_{\{y\in U: |x-y|<R\}}
	\frac{|u_{j,\varepsilon}(x)-u_{j,\varepsilon}(y)|}{|x-y|}
	\rho_{\lambda_j}(x-y)\,dy\,dx\notag\\
	\le &\int_{B(0,\varepsilon)}\eta_\varepsilon(z)
	\int_U\int_{\{y\in U: |x-y|<R\}}
	\frac{|u_{j}(x-z)-u_{j}(y-z)|}{|x-y|}
	\rho_{\lambda_j}(x-y)\,dy\,dx\,dz.
\end{align}

If $z\in B(0,\varepsilon)$ and $|x-y|<R$, then, with
$X=x-z$ and $Y=y-z$,
\[
 d_\Omega(X)\ge d-\varepsilon>\frac{7d}{8}
\]
and hence
\[
 |X-Y|=|x-y|<\frac{\tau d}{8}<\tau d_\Omega(X).
\]
Therefore, by using the fact that $\int_{B(0, \varepsilon)}\eta_\varepsilon=1$ and applying the change of variables $X=x-z$, $Y=y-z$ to \eqref{Fubini-section4},  we obtain
\begin{equation}\label{eq:seq-moll-energy}
 \int_U\int_{\{y\in U:|x-y|<R\}}
 \frac{|u_{j,\varepsilon}(x)-u_{j,\varepsilon}(y)|}{|x-y|}
 \rho_{\lambda_j}(x-y)\,dy\,dx
 \le \E_{\lambda_j,\tau}^{\rho}(u_j,\Omega).
\end{equation}

Let
\[
 W:=\{x\in\Omega:\dist(x,U)<d/2\}.
\]
Since  $\nabla u_{j, \varepsilon}=\nabla\eta_\varepsilon* u_{j}$ and $\nabla u_{j}=\nabla\eta_\varepsilon* u$, by using the facts that  $R+\varepsilon<d/4$ and  $u_j\to u$ in  $L^1(W)$, we have
\begin{equation}\label{nabla-0}
\|\nabla u_{j,\varepsilon}-\nabla u_\varepsilon\|_
{L^\infty(\{x:\dist(x,U)<R\})}
\le
\|\nabla\eta_\varepsilon\|_{L^\infty(\R^n)}
\|u_j-u\|_{L^1(W)}
\to 0
\end{equation}
as $j\to\infty$.
For $x,y\in U$ with $|x-y|<R$, the segment joining $x$ and $y$ is contained in
$\{z:\dist(z,U)<R\}$. Thus
\[
 \frac{|(u_{j,\varepsilon}-u_\varepsilon)(x)
 -(u_{j,\varepsilon}-u_\varepsilon)(y)|}{|x-y|}
 \le
 \|\nabla u_{j,\varepsilon}-\nabla u_\varepsilon\|_
 {L^\infty(\{z:\dist(z,U)<R\})}.
\]
Using \eqref{eq:kernel-mass} and \eqref{nabla-0}, we conclude that
\begin{align*}
&\left|
 \int_U\int_{\{y\in U:|x-y|<R\}}
 \frac{|u_{j,\varepsilon}(x)-u_{j,\varepsilon}(y)|}{|x-y|}
 \rho_{\lambda_j}(x-y)\,dy\,dx\right.\\
&\qquad\qquad\left.
 -
 \int_U\int_{\{y\in U:|x-y|<R\}}
 \frac{|u_\varepsilon(x)-u_\varepsilon(y)|}{|x-y|}
 \rho_{\lambda_j}(x-y)\,dy\,dx
\right|\\
\leq &\mathcal L^n(U) \|\nabla u_{j,\varepsilon}-\nabla u_\varepsilon\|_
{L^\infty(\{z:\dist(z,U)<R\})}\to 0\quad\text{as }j\to\infty.
\end{align*}
Applying Lemma~\ref{lem:smoothtrunc} to the  smooth function
$u_\varepsilon$ gives
\begin{align*}
&\lim_{j\to\infty}
\int_U\int_{\{y\in U:|x-y|<R\}}
\frac{|u_{j,\varepsilon}(x)-u_{j,\varepsilon}(y)|}{|x-y|}
\rho_{\lambda_j}(x-y)\,dy\,dx\\
=&\lim_{j\to\infty}
\int_U\int_{\{y\in U:|x-y|<R\}}
\frac{|u_{\varepsilon}(x)-u_{\varepsilon}(y)|}{|x-y|}
\rho_{\lambda_j}(x-y)\,dy\,dx=
K_{1,n}\int_U|\nabla u_\varepsilon|\,dx.
\end{align*}

Combining this with \eqref{eq:seq-moll-energy} and \eqref{limit-le-L} yields
\[
 K_{1,n}\int_U|\nabla u_\varepsilon|\,dx\le L
 \qquad\text{for every }0<\varepsilon<d/8.
\]
Since $\widetilde u\in L^1(\R^n)$, the standard approximation property of mollifiers gives
\[
u_\varepsilon\rightarrow u
\quad\text{in }L^1(U)
\]
as $\varepsilon\to0^+$. By the lower semicontinuity of the total variation (see \cite[Theorem 5.2]{EvansGariepy}), we have
\[
 K_{1,n}|Du|(U)
 \le
 K_{1,n}\liminf_{\varepsilon\to0^+}\int_U|\nabla u_\varepsilon|\,dx
 \le L.
\]
This gives \eqref{eq:local-liminf} and completes the proof.
\end{proof}

We are ready to prove Theorem~\ref{thm:gamma} and Corollary~\ref{cor:characterization}.

\begin{proof}[Proof of Theorem~\ref{thm:gamma}]
(i)  Suppose that $u_j\to u$ in $L^1(\Omega)$ and set
\[
 L:=\liminf_{j\to\infty}\mathcal F_{\lambda_j}(u_j).
\]
If $L=+\infty$, there is nothing to prove. If $L<+\infty$, it follows from Lemma~\ref{lem:mollify} that for every bounded open set $U\Subset\Omega$,
\[
 K_{1,n}|Du|(U)\le L.
\]
 In particular,
$u\in BV_{\mathrm{loc}}(\Omega)$. Let
$\phi\in C_c^1(\Omega;\R^n)$ with $|\phi|\le1$, and choose a bounded open
$U\Subset\Omega$ with $\supp\phi\subset U$. Then
\[
 \left|\int_\Omega u\,\operatorname{div}\phi\,dx\right|
 \le |Du|(U)\le\frac{L}{K_{1,n}}.
\]
Recall the definition of $|Du|(\Omega)$ in \eqref{eq:BVdual}. Taking the supremum over such $\phi$ and using $u\in L^1(\Omega)$, we obtain
$u\in BV(\Omega)$ and
\[
 K_{1,n}|Du|(\Omega)\le L.
\]

(ii) If $u\in L^1(\Omega)\setminus BV(\Omega)$, $\mathcal F(u)=\infty$ and there is nothing to prove by taking   $u_j=u$. If $u\in BV(\Omega)$, by taking $u_j=u$, we deduce from Theorem~\ref{thm:main} that
\[
 \lim_{j\to\infty}\mathcal F_{\lambda_j}(u_j)
 =K_{1,n}|Du|(\Omega)=\mathcal F(u).
\]
This completes the proof.
\end{proof}

\begin{proof}[Proof of Corollary~\ref{cor:characterization}]
The implication $u\in BV(\Omega)\Rightarrow$ finiteness of the lower limit follows from Theorem~\ref{thm:main}. We now prove the converse. Assume $u\in L^1(\Omega)$ and set
\begin{equation*}\label{eq:Ldef}
	L:=\liminf_{\lambda\to0^+}\E_{\lambda,\tau}^{\rho}(u,\Omega)<\infty.
\end{equation*}
Choose $\lambda_j\to0^+$ as $j\to\infty$ such that
$\E_{\lambda_j,\tau}^{\rho}(u,\Omega)\to L$. By taking $u_j=u$ in part (i) of Theorem~\ref{thm:gamma}, we obtained that $K_{1, n}|Du|(\Omega)\leq L$ and $u\in BV(\Omega)$. Finally,  \eqref{lower-to-limit} follows from Theorem~\ref{thm:main}.
\end{proof}

\section{Two compactness results}\label{sec:compactness}
We first prove local compactness on an arbitrary open set by localizing the functions
away from the boundary and using the classical compactness theorem on bounded Lipschitz domains (cf. \cite[Theorem~1.2]{PonceJEMS2004}). The
global result on bounded open sets then follows from uniform integrability and  Lemma~\ref{lem:local-to-global}.

For $v\in L^1(\R^n)$, set
\[
 \mathcal B_\lambda(v):=
 \int_{\R^n}\int_{\R^n}
 \frac{|v(x)-v(y)|}{|x-y|}\rho_\lambda(x-y)\,dy\,dx.
\]

\begin{lem}\label{lem:cutoff}
Let $\eta\in C_c^\infty(\Omega)$ satisfy $0\le\eta\le1$. Choose $\delta>0$ such
that
\begin{equation}\label{eq:delta-cutoff}
 \delta<\tau\inf_{x\in\supp\eta}d_\Omega(x).
\end{equation}
For $u\in L^1(\Omega)$, let $\widetilde u$ be its zero extension to $\R^n$ and
set $v:=\eta\widetilde u$. Then
\begin{equation}\label{eq:cutoff-estimate}
 \mathcal B_\lambda(v)
 \le
 \E_{\lambda,\tau}^{\rho}(u,\Omega)
 +
 \left(\|\nabla\eta\|_{L^\infty(\R^n)}+\frac{2}{\delta}\right)
 \|u\|_{L^1(\Omega)}.
\end{equation}
\end{lem}

\begin{proof}
Using
\[
 |\eta(x)\widetilde u(x)-\eta(y)\widetilde u(y)|
 \le
 \eta(x)|\widetilde u(x)-\widetilde u(y)|
 +|\widetilde u(y)|\,|\eta(x)-\eta(y)|,
\]
we write $\mathcal B_\lambda(v)\le I_1+I_2$. Since $\rho_\lambda$ satisfies \eqref{eq:kernel-mass} and
$$|\eta(x)-\eta(y)|\le\|\nabla\eta\|_{L^\infty(\mathbb R^n)}|x-y|,$$
 we deduce that
\begin{align}\label{sec5-I2}
 I_2&=\int_{\R^n}\int_{\R^n}
 \frac{|\widetilde u(y)|\,|\eta(x)-\eta(y)|}{|x-y|}\rho_\lambda(x-y)\,dy\,dx\notag\\
 &\le\|\nabla\eta\|_{L^\infty(\mathbb R^n)}\|\widetilde u\|_{L^1(\mathbb R^n)}=\|\nabla\eta\|_{L^\infty(\mathbb R^n)}\|u\|_{L^1(\Omega)}.
\end{align}

For the estimate of $I_1$, we split $\mathbb R^n\times \mathbb R^n$ into the regions $\{(x, y)\in \mathbb R^n\times \mathbb R^n: |x-y|<\delta\}$ and $\{(x, y)\in \mathbb R^n\times \mathbb R^n: |x-y|\ge\delta\}$. On the
first region, $\eta(x)\ne0$ implies $x\in\supp\eta$, and then
\eqref{eq:delta-cutoff} gives
\[
 |x-y|<\delta<\tau d_\Omega(x).
\]
In particular, if $\eta(x)\not=0$, then $y\in\Omega$ with $|x-y|\leq \tau d_\Omega(x)$. Hence
\[
 \iint_{\{(x, y)\in \mathbb R^n\times \mathbb R^n: |x-y|<\delta\}}
 \eta(x)\frac{|\widetilde u(x)-\widetilde u(y)|}{|x-y|}
 \rho_\lambda(x-y)\,dy\,dx
 \le \E_{\lambda,\tau}^{\rho}(u,\Omega).
\]
On the second region, since $\rho_\lambda$ satisfies \eqref{eq:kernel-mass} and $0\leq \eta\leq 1$, we obtain
\begin{align*}
&\iint_{\{(x, y)\in \mathbb R^n\times \mathbb R^n: |x-y|\ge\delta\}}
 \eta(x)\frac{|\widetilde u(x)-\widetilde u(y)|}{|x-y|}
 \rho_\lambda(x-y)\,dy\,dx\\
&\le\frac1\delta
 \iint_{\mathbb R^n\times\mathbb R^n} \eta(x)\bigl(|\widetilde u(x)|+|\widetilde u(y)|\bigr)
 \rho_\lambda(x-y)\,dy\,dx
 \le\frac{2}{\delta}\|u\|_{L^1(\Omega)}.
\end{align*}
Consequently, 
$$I_1=\int_{\R^n}\int_{\R^n}
\eta(x)\frac{|\widetilde u(x)-\widetilde u(y)|}{|x-y|}\rho_\lambda(x-y)\,dy\,dx \leq \E_{\lambda,\tau}^{\rho}(u,\Omega)+ \frac{2}{\delta}\|u\|_{L^1(\Omega)}.$$
Combining this with \eqref{sec5-I2} proves \eqref{eq:cutoff-estimate}, and hence the proof is complete.
\end{proof}

\begin{proof}[Proof of Theorem~\ref{thm:compactness-local}]
Fix a bounded open set $U\Subset\Omega$. Choose
$\eta\in C_c^\infty(\Omega)$ such that $0\le\eta\le1$ and $\eta=1$ on a
neighborhood of $\overline U$. By compactness of $\supp\eta$ in $\Omega$, a
number $\delta>0$ satisfying \eqref{eq:delta-cutoff} exists. Let
$\widetilde u_j$ be the zero extension of $u_j$ and set
\[
 v_j:=\eta\widetilde u_j.
\]
By Lemma~\ref{lem:cutoff} and \eqref{eq:compactness-assumption}, we have
\[
 \sup_{j\geq 1}\left(
 \|v_j\|_{L^1(\R^n)}+\mathcal B_{\lambda_j}(v_j)
 \right)<\infty.
\]
Choose a bounded ball $B$ containing $\supp\eta$. Since $n\ge2$, the sequence
$\rho_{\lambda_j}$ satisfies the hypotheses of
\cite[Theorem~1.2]{PonceJEMS2004}. Because the bounded ball $B$ is a Lipschitz domain, it follows from \cite[Theorem~1.2]{PonceJEMS2004} that
$(v_j)$ is relatively compact in $L^1(B)$. Since $v_j=u_j$
on $U$, the sequence $(u_j)$ is relatively compact in $L^1(U)$.

Let $(U_m)$ be a bounded exhaustion of $\Omega$ with
$U_m\Subset U_{m+1}\Subset\Omega$. A diagonal argument gives a subsequence $(u_{j_k})$ and $u\in L^1_{\mathrm{loc}}(\Omega)$ such that
\[
 u_{j_k}\to u\qquad\text{in }L^1_{\mathrm{loc}}(\Omega).
\]
Passing to a further subsequence, we may assume $u_{j_k}\to u$ almost everywhere in
$\Omega$. By Fatou's lemma and \eqref{eq:compactness-assumption}, we have
\[
 \|u\|_{L^1(\Omega)}\le\liminf_{j_k\to\infty}\|u_{j_k}\|_{L^1(\Omega)}<\infty.
\]

Set
\[
 L:=\liminf_{j_k\to\infty}\E_{\lambda_{j_k},\tau}^{\rho}(u_{j_k},\Omega).
\]
It follows from Lemma~\ref{lem:mollify}  that for every bounded open $V\Subset\Omega$,
\[
 K_{1,n}|Du|(V)\le L.
\]
As in the proof of Theorem~\ref{thm:gamma}, the definition of $|Du|(\Omega)$ in \eqref{eq:BVdual}  yields $u\in BV(\Omega)$ and
\[
 K_{1,n}|Du|(\Omega)\le L.
\]
This gives \eqref{eq:compactness-liminf} and completes the proof.
\end{proof}

\begin{proof}[Proof of Theorem~\ref{thm:compactness}]
By Theorem~\ref{thm:compactness-local}, after passing to a subsequence $(u_{j_k})$, there is
$u\in BV(\Omega)$ such that
\[
 u_{j_k}\to u\qquad\text{in }L^1_{\mathrm{loc}}(\Omega),
\]
and \eqref{eq:compactness-liminf} holds. Since $\Omega$ is bounded and $(u_j)$
is uniformly integrable in $L^1(\Omega)$, Lemma~\ref{lem:local-to-global} implies
\[
 u_{j_k}\to u\qquad\text{in }L^1(\Omega),
\]
which completes the proof.
\end{proof}

\section{Examples}\label{sec:counterexample}
In this section, we present the following three counterexamples.

\subsection{Failure of the classical BBM formula on arbitrary open sets}\label{subsec:slit-counterexample}
We show that the standard fractional Sobolev energy in
\eqref{eq:classical-bbm} does not yield the BBM formula on arbitrary open sets.

Let $\{\rho_\lambda\}_{\lambda>0}$ be any family of radial mollifiers satisfying
\eqref{eq:kernel-mass} and \eqref{eq:kernel-tail}. For every $n\ge2$, the
following slit-domain construction gives a bounded domain $\Omega\subset\R^n$
and a function $u\in BV(\Omega)$ for which
\begin{equation}\label{eq:counter}
	\lim_{\lambda\to0^+}
	\int_\Omega\int_\Omega
	\frac{|u(x)-u(y)|}{|x-y|}\rho_\lambda(x-y)\,dx\,dy
	\ne K_{1,n}|Du|(\Omega).
\end{equation}

Let
\[
Q:=(-1,1)^n
\]
and let $S$ be the half-hyperplane  slit in $Q$ defined by
\[
S:=\{x=(x_1,\ldots,x_n)\in Q:x_1=0,\ x_2\le0\}.
\]
Set $\Omega:=Q\setminus S$. Clearly, $\Omega$ is a  bounded domain in $\mathbb R^n$.

Define 
\[
E:=\{x=(x_1,\ldots,x_n)\in \mathbb R^n: x_1>0\}\quad \text{and} \quad u(x):=\mathbf1_{E}(x),\  x\in\mathbb R^n.
\]
Then $E$ has locally finite perimeter in $\mathbb R^n$. Moreover,  $|Du|=|D\mathbf1_{E}|$ coincides with the $(n-1)$-dimensional Hausdorff measure restricted on the reduced boundary $\partial^* E$; see \cite[Theorem 5.15]{EvansGariepy}. 

Because $\partial^* E=\{x=(x_1,\ldots,x_n)\in \mathbb R^n: x_1=0\}$, we have
$$\partial^* E\cap \Omega=\{x\in Q:x_1=0,\ x_2>0\}\quad \text{and}\quad \partial^* E\cap Q=\{x\in Q:x_1=0\}.$$
Therefore,
\begin{equation}\label{eq:DuOmega}
	|Du|(\Omega)=\Hn(\partial^* E\cap \Omega)=2^{n-2}\quad \text{and} \quad  |Du|(Q)=\Hn(\partial^* E\cap Q)=2^{n-1}. 
\end{equation}
Then we know that $u\in BV(\Omega)$ and $u\in BV(Q)$.

Since the slit $S$ has zero $n$-dimensional Lebesgue measure, for every $\lambda>0$, we have
\begin{equation}\label{eq:same-integral}
	\int_\Omega\int_\Omega
	\frac{|u(x)-u(y)|}{|x-y|}\rho_\lambda(x-y)\,dx\,dy
	=\int_Q\int_Q
	\frac{|u(x)-u(y)|}{|x-y|}\rho_\lambda(x-y)\,dx\,dy.
\end{equation}
Because  $Q$ is a Lipschitz domain, D\'avila's theorem \cite[Theorem 1]{Davila} yields
\[
\lim_{\lambda\to0^+}
\int_Q\int_Q
\frac{|u(x)-u(y)|}{|x-y|}\rho_\lambda(x-y)\,dx\,dy
=K_{1,n}|Du|(Q)=K_{1,n}2^{n-1}.
\]
Combining this with \eqref{eq:DuOmega} and \eqref{eq:same-integral}, we obtain
\[
\lim_{\lambda\to0^+}
\int_\Omega\int_\Omega
\frac{|u(x)-u(y)|}{|x-y|}\rho_\lambda(x-y)\,dx\,dy
=2K_{1,n}|Du|(\Omega).
\]
Thus \eqref{eq:counter} holds, and the  BBM limit of standard fractional Sobolev energy in \eqref{eq:classical-bbm} disagrees with $K_{1,n}|Du|(\Omega)$.

\subsection{Failure of global compactness on unbounded domains}\label{subsec:unbounded-counterexample}
The boundedness of $\Omega$ in the global compactness theorem, Theorem~\ref{thm:compactness}, cannot be omitted.
Take $\Omega=\R^n$ and let $0\ne v\in C_c^\infty(\R^n)$. Choose a sequence
$\{a_j\}_{j=1}^\infty\subset\R^n$ with $|a_j-a_k|$ larger than twice the diameter of $\supp v$ when
$j\ne k$, and set
\[
 u_j(x):=v(x-a_j).
\]
By translation invariance, $(u_j)$ is uniformly integrable in $L^1(\R^n)$ and
$\|u_j\|_{L^1(\mathbb R^n)}=\|v\|_{L^1(\mathbb R^n)}$. Moreover, for every sequence $\lambda_j\to 0^+$ as $j\to \infty$,  
$$\E_{\lambda_j,\tau}^{\rho}(u_j,\R^n)=\E_{\lambda_j,\tau}^{\rho}(v,\R^n).$$
Then it follows from Proposition~\ref{prop:Ele} that
$$\sup_{j\geq 1}\E_{\lambda_j,\tau}^{\rho}(u_j,\R^n)=\sup_{j\geq 1}\E_{\lambda_j,\tau}^{\rho}(v,\R^n)
\le K_{1,n}|Dv|(\R^n)<\infty.$$
Combining this with $\|u_j\|_{L^1(\mathbb R^n)}=\|v\|_{L^1(\mathbb R^n)}$ yields that $(u_j)$ satisfies \eqref{eq:compactness-assumption}.
On the other hand, the supports of $(u_j)$ are pairwise disjoint, so
\[
 \|u_j-u_k\|_{L^1(\R^n)}=2\|v\|_{L^1(\R^n)}
 \qquad \text{whenever}\  j\ne k.
\]
Thus no subsequence of $(u_j)$ can converge strongly in $L^1(\R^n)$.

\subsection{Failure of global compactness without uniform integrability}\label{subsec:ui-counterexample}
The uniform integrability assumption in Theorem~\ref{thm:compactness} cannot in
general be omitted, even if $\Omega$ is bounded and connected. For brevity, we present an explicit construction in $\mathbb{R}^2$, which can be readily generalized to higher dimensions.

Let 
\[
 Q:=(0,1)^2.
\]
For $k\geq 1$, set 
$$I_k=(2^{-k}-2^{-4k}, 2^{-k}+2^{-4k})\quad\text{and} \quad R_k=I_k\times(-1, 1).$$
Then the sets $R_k$ are pairwise disjoint. Define
$$\Omega=Q\cup \bigcup_{k\ge1} R_k.$$
Then $\Omega$ is bounded and connected. Let 
$$E_k=I_k\times (-1, -1/2).$$
Since the only part of the reduced boundary of $E_k$ lying inside $\Omega$ is $I_k\times \{-1/2\}$, we have
$$P(E_k;\Omega)=\mathcal H^{1}(I_k)=2^{-4k+1}.$$
Moreover, 
$$\mathcal L^2(E_k)=\frac{1}{2} \mathcal H^{1}(I_k)=2^{-4k}.$$
Now define
\[
 u_k:=\frac{\mathbf1_{E_k}}{\mathcal L^2(E_k)}.
\]
Then
\[
 \|u_k\|_{L^1(\Omega)}=1,
 \qquad
 |Du_k|(\Omega)=\frac{P(E_k;\Omega)}{\mathcal L^2(E_k)} =2.
\]
For every sequence $\lambda_k\to 0^+$ as $k\to \infty$,  
by Proposition~\ref{prop:Ele},   we have
\[
 \E_{\lambda_k,\tau}^{\rho}(u_k,\Omega)
 \le K_{1,2}|Du_k|(\Omega)\le 2K_{1,2}.
\]
Thus $(u_k)$ satisfies \eqref{eq:compactness-assumption}.  However, the sets $E_k\subset R_k$ are pairwise
disjoint, and hence
\[
 \|u_j-u_k\|_{L^1(\Omega)}=2
 \qquad \text{whenever}\  j\ne k.
\]
Therefore, $(u_k)$ has no strongly convergent subsequence in $L^1(\Omega)$.

This construction also shows exactly where uniform integrability fails:
although $\mathcal L^2(E_k)\to 0$ as $k\to\infty$,
\[
 \int_{E_k}|u_k|\,dx=1
 \quad\text{for every }k\geq 1.
\]

\vspace{1cm}

\noindent\textbf{Acknowledgments.} 
Jiang Li was partly supported by NSFC under Grant No. 12571081, Grant No. 12371071 and the Key Project of NSF of Hunan Province under Grant No. 2026JJ30002.  Zhuang Wang was partly supported by NSF of Hunan Province under Grant No. 2024JJ6299, the Scientific Research Fund of Hunan Provincial Education Department  under Project No. 25B0095, and  NSFC under Grant No. 12101226.

\end{document}